\documentclass{amsart}

\usepackage[T1]{fontenc}
\usepackage{times}
\usepackage{amssymb,verbatim}
\theoremstyle{plain}
\usepackage[T1]{fontenc}
\usepackage[utf8]{inputenc}
\usepackage[english]{babel}
\usepackage{amssymb}
\usepackage{amsthm}
\usepackage{graphics}
\usepackage{amsmath}
\usepackage{amstext}
\usepackage{multirow}
\usepackage{graphicx}
\usepackage{amsmath}
\usepackage{amssymb}
\usepackage{amsthm}
\usepackage{amstext}
\usepackage{mathrsfs}
\usepackage{amscd}
\usepackage{mathtools} 
\usepackage{hyperref}
\usepackage{tikz}
\usetikzlibrary{matrix}
\usepackage[all]{xy}
\usepackage{algorithm}
\usepackage{algorithmic}

\newtheorem{thm}{Theorem}[section]
\newtheorem*{conj*}{Conjecture}

\newtheorem{dfn}[thm]{Definition}
\newtheorem{lemma}[thm]{Lemma}
\newtheorem{prop}[thm]{Proposition}
\newtheorem{cor}[thm]{Corollary}

\newtheorem{THM}{Theorem}

\theoremstyle{remark}

\newtheorem{example}[thm]{Example}
\newtheorem{remark}[thm]{Remark}

\newcommand{\Foliation}[2][3]{\textsf{Fol}_{#2}(\mathbb P^{#1})}

\newcommand{\Fo}{\textsf{Fol}(d)}
\newcommand{\Fol}[2]{\textsf{Fol}({#2})}

\newcommand{\mb}{\mathbb}
\newcommand{\mc}{\mathcal}
\newcommand{\R}{\mb R}
\newcommand{\C}{\mb C}

\newcommand{\Z}{\mb Z}
\newcommand{\Q}{\mb Q}
\newcommand{\N}{\mb N}

\newcommand{\w}{\ell}
\newcommand{\p}{k}

\newcommand{\Proj}{\mathbb{P}} 
\newcommand{\Co}{\mathcal{C}}
\newcommand{\Ol}{\mathcal{O}}

\newcommand{\Dl}{\mathcal{D}}
\newcommand{\Si}{\mathcal{S}}

\newcommand{\F}{\mc F}
\newcommand{\G}{\mc G}
\newcommand{\Ho}{\mc H}
\newcommand{\Ro}{\mathcal{R}}

\DeclareMathOperator{\sing}{sing}

\DeclareMathOperator{\Pic}{Pic}

\DeclareMathOperator{\mdc}{mdc}
\DeclareMathOperator{\lcm}{lcm}

\newcommand{\KF}{{K_\F}}

\newcommand{\TF}{{T_\F}}

\newcommand{\NF}{{N_\F}}
\newcommand{\NFo}{{N_{\F_0}}}
\newcommand{\NFone}{{N_{\F_1}}}
\renewcommand{\NG}{{N_\G}}

\newcommand{\TP}{{T_{\Proj}}}

\numberwithin{equation}{section}

\usepackage{mathtools}
\usepackage{ifmtarg}

\addtocounter{section}{0}             
\numberwithin{equation}{section}       

\title[Density of Foliations
Without Algebraic Solutions   ]
      {On The Density of Foliations
Without Algebraic Solutions on Weighted Projective Planes }

\author{Ruben Lizarbe}

\begin{document}

\begin{abstract}
We prove that a generic holomorphic foliation on a weighted projective 
plane has no algebraic solutions when the
degree is big enough. We also prove an analogous result for
foliations on Hirzebruch surfaces.
\end{abstract}

\maketitle

\setcounter{tocdepth}{1}


\section{Introduction}

One of the most important results of Jouanolou's celebrated monograph
\cite{Jou} states that the set of holomorphic foliations on the
complex projective plane $\Proj^2$ of degree at least 2 which do not
have an algebraic solution is dense in the space of foliations.
This result for one-dimensional holomorphic foliations on $\Proj^n$
was proved by Lins Neto and Soares in \cite{LinsSoares}. In \cite{Cou}
the authors prove a generalization of Jouanolou's result for 
one-dimensional foliations over any smooth projective variety. On the
other hand, in \cite{Jorge} the author gives a different proof of
Jouanolou's theorem by restricting the ideas of \cite{Cou} to $\Proj^2$. In 
\cite{Falla}, one can find  three versions of
Jouanolou's Theorem for second order differential equations on
$\Proj^2$, for $k$-webs (first order differential equations) on
$\Proj^2$ and for webs with sufficiently ample normal bundle on
arbitrary projective surfaces.

The main theorem of this work provides an analogous version of Jouanolou's
Theorem for  foliations on weighted projective planes.

\begin{THM}\label{T:Weighted general}
Let $l_0,l_1,l_2$ be pairwise coprimes numbers with $1\leq l_0\le l_1 \le l_2$.
A generic foliation with normal $\Q$-bundle of degree $d$ on the weighted 
projective plane $\Proj(l_0,l_1,l_2)$ does not admit any invariant algebraic curve if
$d\geq l_0l_1l_2+l_0l_1+ 2l_2$.
\end{THM}

The bound above is not sharp. When $l_0=l_1=1$ and $l_2=k>1$, we denote $\Proj(1,1,k)$ 
by $\Proj^2_{k}$, for which there is 
a more precise version of Theorem \ref{T:Weighted general}  that 
is sharp.

\begin{THM}\label{T:Weighted special}
A generic foliation with normal $\Q$-bundle of degree $d$ on
$\Proj^2_k$ with $k\geq2$ does not admit any invariant
algebraic curve if $d\geq 2k+1$. Moreover, if $d<2k+1$ any
foliation with normal $\Q$-bundle of degree $d$ on
$\Proj^2_k$ admits some invariant algebraic curve.
\end{THM}

In both statements,  by generic we mean the set of foliations
having no invariant curves is the complement of a
countable union of algebraic closed proper subsets.

It is worth mentioning that there are several works about foliations on weighted projective spaces, see 
\cite{CoRoSoa}, \cite{Correa} and \cite{MaCoRo}.

It is well known that the minimal resolution of the
weighted projective planes $\Proj^2_k$, $k\geq 2$, are
the Hirzebruch surfaces
$\mb F_k=\Proj(\Ol_{\Proj^1}\oplus\Ol_{\Proj^1}(k))$ (see
\cite{Reid}). Our next result is a generalized version of
Jouanolou's Theorem for foliations on Hirzebruch surfaces.

\begin{THM}\label{T:Hirzebruch}
A generic foliation with normal bundle of bidegree $(a,b)$ on
$\mb F_{k}$ does not admit any invariant algebraic curve if $b\geq3$
and $a\geq kb+2$. Moreover, if $a<kb+2$ or $b<3$ then any
foliation with normal bundle of bidegree $(a,b)$ on $\mb F_{k}$
admits some invariant algebraic curve.
\end{THM}

\subsection{Organization of the paper}
Section \ref{foliations}, we introduce the notions of weighted projective planes and foliations 
on weighted projective planes. 
Section \ref{algebraicleaves}, we study foliations on weighted projective planes with some algebraic solution. 
We show that the low degree foliations correspond to rational, logarithmic and Riccati foliations.  
Section \ref{Theorem}, we prove our first and second main results about the version of 
Jouanolou's theorem for foliations on weighted projective planes.
Section \ref{Hirzebruch}, we introduce the notions of foliations 
on Hirzebruch surfaces and prove our third main result Theorem \ref{T:Hirzebruch}. 


\section{Holomorphic foliations on $\Proj(l_0,l_1,l_2)$}\label{foliations}

\subsection{Weighted projective planes}

Let $\w=(l_0,l_1,l_2)$ be a vector of positive integers with
$l_0,l_1,l_2$ pairwise coprimes, and call $\w$ a {\it weighted
vector}. Consider the $\C^*$-action on $\C^{3}\backslash \{0\}$
given by
\begin{equation}\label{E:action}
 t\cdot(x_0,x_1,x_2)=(t^{l_0}x_0,t^{l_1}x_1,t^{l_2}x_2),
\end{equation}
where $t\in \C^*$ and $(x_0,x_1,x_2)\in \C^3\backslash \{0\}$. 
The {\em weighted projective plane} of type $(l_0,l_1,l_2)$ is the quotient space
$\Proj(l_0,l_1,l_2)=(\C^{3}\backslash \{0\}/\sim)$, induced by the above action. 
Consider the decomposition $\Proj(l_0,l_1,l_2)=U_0\cup U_1\cup U_2$,
where $U_i$, $0\leq i \leq 2$, is the open set consisting of all elements
$[x_0:x_1:x_2]$ with $x_i\neq 0$. For a fixed $i$, let $j<k$ be the elements of $\{0,1,2\}\backslash\{i\}$
and denote by $\mu_{l_i}\subset \C^*$ the
subgroup of $l_i$-th roots of the unity, which acts on 
$\C^2$ by $g\cdot (x,y)=(g^{l_j}x,g^{l_k}y)$, for any $g\in \mu_{l_i}$.  
We can define an isomorphism
$\psi_i:U_i\rightarrow\C^2/ \mu_{l_i}$ by
\begin{equation}\label{E:change}
\psi_i([x_0:x_1:x_2])=\left(\frac{x_j}{x_i^{l_j/l_i}},\frac{x_k}{x_i^{l_k/l_i}}\right)_{l_i}.
\end{equation}
Observe that  $\psi_i$ is independent of the choice of the $l_i$-th root of $x_i$.

\begin{remark}\label{R:projective plane}
When $l_0=l_1=l_2=1$ one obtains the usual complex projective plane, in which case we write $\Proj^2$ 
instead of $\Proj(1,1,1)$.
\end{remark}
We can also view the weighted projective plane as follows. 
Consider the group $G=\mu_{l_0}\times \mu_{l_1}\times \mu_{l_2}$, acting 
on $\Proj^2$ by $(g_0,g_1,g_2)\cdot[x_0:x_1:x_2]=[g_0x_0:g_1x_1:g_2x_2]$. 
The natural map $\varphi: \C^3\rightarrow \C^3$ given by
\begin{equation}\label{E:map}
\varphi(x_0,x_1,x_2)=(x_0^{l_0},x_1^{l_1},x_2^{l_2}),
\end{equation}
induces an isomorphism from the set of orbits 
$\Proj^2/G$ to $\Proj(l_0,l_1,l_2)$. 

From now on we will assume that $1\leq l_0 \leq l_1 \leq l_2$. For simplicity of notation, 
we write $\Proj$ instead of $\Proj(l_0,l_1,l_2)$ as in the terminology of
\cite{Dolgachev}. The weighted projective plane $\Proj$ is a singular surface with at worst 
quotient singularities.
\subsection{Twisted differentials defining foliations}

If we denote by $\Omega^{[1]}_{\Proj}$ the sheaf of reflexive differentials on $\Proj$ ($\Omega^{[1]}_{\Proj}= 
\overline{\Omega^1_{\Proj}}$ in the notation of \cite[Section 2.1]{Dolgachev}) and by $\Ol_{\Proj}(d)$ the sheaf 
of $\Ol_{\Proj}$-modules associated to the module of quasi-homogeneous polynomials of degree $d$ 
(\cite[Section 1.4]{Dolgachev} ) then  
we can use the following twisting of Euler's sequence \cite[Section 2.1]{Dolgachev},
\[
    0 \to \Omega^{[1]}_{\Proj}(d) \to   \bigoplus _{i=0}^{2}\mathcal O_{\mathbb P}(d-l_i) \to 
    \mathcal O_{\mathbb P}(d) \to 0 \, ,
\]
to identify $H^0(\Proj, \Omega^{[1]}_{\Proj}\otimes \Ol_{\Proj}(d))=H^0(\mathbb P, \Omega^{[1]}_{\mathbb P}(d))$ with the $\mathbb C$-vector space of quasi-homogenous polynomial $1$-forms
\[
      A_0(x_0,x_1, x_2) dx_0+ A_1(x_0,x_1, x_2) dx_1+ A_2(x_0,x_1, x_2) dx_2
\]
with quasi-homogeneous coefficients $A_i$ of degree $d-l_i$, for $i=0,1,2$,  which are annihilated by the weighted radial vector field
\[
    R = l_0x_0 \frac{\partial}{\partial x_0}+ l_1x_1 \frac{\partial}{\partial x_1}+ l_2x_2 \frac{\partial}{\partial x_2} \, .
\]


A {\it foliation} $\F$ on $\Proj$ is defined by a section $\omega$ of $\Omega^{[1]}_{\Proj}\otimes \Ol_{\Proj}(d)$ 
with {\it normal sheaf} $\NF=\Ol_{\Proj}(d)$.  

The singular set of $\F$, denoted by $\sing(\F)$ is the subset of $\Proj$ 
 formed by zeros of $\omega$. A foliation $\F$ is called {\it saturated} if $\sing(\F)$ is finite.
\begin{remark}\label{R:degreefol}
In the case $\w=(1,1,1)$, we have $\deg(N\F)=\deg(\F)+2,$
where $\deg(\F)$ is the number of tangencies with general line. 
\end{remark}

By duality, we also have the following twisting of Euler's sequence 
\[
    0 \to \mathcal O_{\mathbb P}(d-|\w|)  \to   \bigoplus _{i=0}^{2}\mathcal O_{\mathbb P}(d-|\w|+l_i) \to 
        \TP(d-|\w|) \to 0 \, ,
\]
where $|\w|=l_0+l_1+l_2$. Then a foliation $\F$ of normal degree $d$ on $\Proj$ 
can also be given by a quasi-homogeneous vector field 
$X\in H^0(\Proj, \TP\otimes \Ol_{\Proj}(d-|\w|))=H^0(\mathbb P, \TP(d-|\w|))$, that is,   
\[
X=B_0\frac{\partial}{\partial x_0}+B_1\frac{\partial}{\partial x_1}+
B_2\frac{\partial}{\partial x_2}, 
\]
which is not a multiple of $R$, where $B_i$ are
quasi-homogeneous polynomials of degree $d-|\w|+l_i$, respectively. By contracting the volume 
form $dx_0\wedge dx_1\wedge dx_2$ with the vector fields $X$ and $R$, we obtain a 1-form  
\[
\omega=i_Xi_Rdx_0\wedge dx_1\wedge dx_2 \in H^0(\mathbb P, \Omega^{[1]}_{\mathbb P}(d)),
\]
that defines the foliation $\F$. This vector field $X$ also induces the sheaf $\KF=\Ol_{\Proj}(d-|\w|),$ 
called \textit{the cotangent or canonical sheaf} of $\F$. The dual $\KF$ is $\TF$ called 
{\it the tangent sheaf of $\F$}. 
Note that $\NF$, $\KF$ and $\TF$  are no longer genuine line bundles but are elements of 
$\Pic(\Proj)\otimes\Q $. 

\subsection{Space of foliations and interpretation}

Two 1-forms $\omega$ and $\omega'$ define the same foliation if and only if they differ 
by multiplication by a nonzero complex constant. Therefore,      
the {\it space of foliations with normal $\Q$-bundle of
degree $d$} on $\Proj$ is denoted by 
\[
\Fo:= \Proj H^0(\Proj,\Omega^{[1]}_{\Proj}(d)).
\]
We denote by $\Foliation[2]{d}^{G}$ the space of $G$-invariant foliations of degree $d$ on $\Proj^2$. 
Note that the pullback of a quasi-homogeneous 1-form by $\varphi$ (as in 
\ref{E:map}) is a homogeneous 1-form invariant by $G$. 
Consider the natural map:   

\begin{equation}\label{E:isomorphism}
\begin{array}{rcl}
\varphi^*: \Fol_{d+2}&\rightarrow& \Foliation[2]{d}^{G} \\
\,[\omega] \quad &\mapsto &[\varphi^*\omega] .
\end{array}
\end{equation}

It is a direct computation to show that the following holds.  

\begin{lemma}\label{L:isomorphism} The map $\varphi^*$ is an isomorphism.  
\end{lemma}
This means that a foliation on $\Proj$ is identified with a foliation on $\Proj^2$ invariant by $G$. Therefore, 
$\Foliation[2]{d}^{G}$ is an irreducible closed set of the space of foliations of degree $d$ on $\Proj^2$. 

Denoting by $S_n\subset \C[x_0,x_1,x_2]$ the space of 
quasi-homogeneous polynomials of degree $n$, we 
have the na\-tu\-ral application 
$\phi_n:\Proj(S_n)\times \Fol_{d-n} \rightarrow \Fol_d$,
$([F],[\omega]) \mapsto [F\omega],$
for $1\leq n\leq d$. 
The set of the saturated foliations in 
$\Fo$ is equal to the complement of $\displaystyle\cup_{1\leq n\leq d} Im (\phi_n)$, 
which is an open Zariski set in $\Fo$. 
 Note that this open set can be empty, as illustrated by the 
 following example.
 \begin{example} \label{E:particular case} 
 {\rm Take $\w=(1,1,\p)$, $\p\geq 2$. Looking at the coefficients of the forms 
 we can prove that $\Fo=Im (\phi_{d-2})$ for all $2< d \leq \p$.}
 \end{example}
 
 \subsection{Invariant algebraic curves}

Let $\F=[\omega]\in \Fo$ and $C\subset \Proj$ be an
irreducible algebraic curve of degree $n$, that is, defined by a  
quasi-homogeneous polynomial $F$ of degree $n$.  
We say that $C$ is $\F$-invariant if $i^*\omega\equiv0$, where 
$i$ denotes the inclusion of the intersection of
smooth loci of both $C$ and $\Proj$. 
Equivalently, $C$ is $\F$-invariant if 
there exists a quasi-homogeneous $2$-form $\Theta_{F}$ of
degree $n$ such that
\begin{equation}\label{E:invariant}
 \omega\wedge dF-F\Theta_{F}=0.
\end{equation}
\begin{remark}
An important fact about equation $(\ref{E:invariant})$ is that it still
works for reducible curves, {\it i.e.}, if the decomposition of $F$ is $F^{n_1}_1\ldots F^{n_r}_r$, then the equation
(\ref{E:invariant}) holds if and only if each irreducible factor $F_j$
defines an $\F$-invariant curve.
\end{remark}

Let  $p_0:=[1:0:0]$, $p_1:=[0:1:0]$, $p_2:=[0:0:1]$ and consider
the following sets
\begin{align*}
  \Co_{n}(d)&:=\{\F\in\Fo|\,
  \mbox{ there is a $\F$-invariant algebraic curve of degree $n$}\},\\
   \Dl_{n}(d)&:=\left\{\begin{array}{rl}
   (x,\F)\in\Proj\times \Fo|&\mbox{$x$ belongs to some $\F$-invariant algebraic }\\
  &\quad \mbox{ curve of degree $n$}
   \end{array}\right\},
\end{align*}
  and for each $i\in\{0,1,2\}$
\[  
  \Co^{p_i}_{n}(d):=\{\F\in\Fo|\quad
  \mbox{$p_i$ belongs to some $\F$-invariant algebraic curve of degree $n$}\}.
\]
The following lemma will be used in \S \ref{S:wirthoutleaves}. 
\begin{lemma}\label{L:closedset}
 The sets $\Co_{n}(d)$, $\Dl_{n}(d)$ and $\Co^{p_i}_{n}(d)$ are closed 
 for every $n$ and $i$.
\end{lemma}
\begin{proof}  The same argument of \cite[Proposition 1]{Jorge} shows that 
$\Co_{n}(d)$ and $\Dl_{n}(d)$ are closed sets. In order to finish we just note 
that $\pi_2(\pi_1^{-1}(p_i)\cap \Dl_{n}(d))=\Co^{p_i}_{n}(d)$, where $\pi_1$ and $\pi_2$ 
are canonical projections of $\Proj\times \Fo$. 
\end{proof}


\subsection{Saturated foliations} 
Let $\F$ be a saturated foliation on $\Proj$ and $p\in \sing(\F)$.
\begin{dfn} The \textit{multiplicity (or Milnor number) of $\F$ at $p$} which is 
denoted by $m(\F,p)$,  is defined as: 
\begin{enumerate}
\item If $p\notin\, \sing(\Proj)$, and $\F$ is locally given by $\omega=A(x,y)\,dx+B(x,y)\,dy$, then  
$$m(\F,p):=dim_{\C}\frac{\Ol_p}{\langle A,B\rangle},$$
where $\Ol_p$ is the local algebra of $\Proj$ at $p$. 
\item If $p\,\in\,Sing(\Proj)$, 
then  
\[
m(\F,p):=\frac{m(\G,(0,0))}{r},
\]
in which $U\simeq  \C^2/\mu_{r}$ is a neighborhood of $p$   
and $\G$ is the lifting of $\F_{|_{U}}$ to $\C^2$.
\end{enumerate}
Therefore we can define:
\[
m(\F):=\displaystyle\sum_{p\in\,\Proj}m(\F,p).
\]
\end{dfn}
In the index theory of foliations on complex regular surfaces there are other indices such as tangency, 
vanishing, Camacho-Sad and Bam-Bott, for a good reference see \cite{Brun1}. This theory is also valid on weighted projective planes since they are singular surfaces with at worst quotient singularities, for this see \cite{Brun2} 
and \cite{CoRoSoa}. As an adaptation of \cite[Proposition 1, page 21]{Brun1} we have the following formula
\[
m(\F) = c_2 (\Proj) -\TF\cdot \NF,
\]
where $c_2$ is the 2- Chern class of $\Proj$.

\begin{prop}\label{P:Milnornumber}
Let $\F$ be a foliation of normal degree $d$ on $\Proj$. Then
\begin{equation}\label{E:formula}
l_0l_1l_2m(\F)=
l_0l_1+l_0l_2+l_1l_2+(d-|\w|)d.
\end{equation}
\end{prop}

\begin{proof}  
It follows from $c_2(\Proj)=\frac{l_0l_1+l_0l_2+l_1l_2}{l_0l_1l_2}$, $\TF=\Ol_{\Proj}(|\w|-d)$, 
and $\NF=\Ol_{\Proj}(d)$. 
\end{proof}
It is worth mentioning that another proofs can be found in \cite[Proposition 3.2]{CoRoSoa} and 
\cite[Proposition 1.7.6]{Lizarbe14}. 

Again for a fixed $i$, let $j<k$ be the elements of $\{0,1,2\}\backslash\{i\}$. 
We will work on an affine coordinate system $(x, y)\in \C^2$ where the action of $\mu_{l_i}$ on 
$\C^2$ defines $U_i=\C^2/\mu_{l_i}\subset \Proj$.

In this coordinate system, a foliation $\F$ of normal degree $d$ on $\Proj$ is defined by a 1-form 
$\omega$ that can be written as a sum of quasi-homogeneous 1-forms
\begin{equation}\label{E:decomposition}
 \omega = \omega_r + \omega_{r+l_i} + \omega_{r+2l_i} + \ldots+ \omega_{d-l_i}+\omega_{d}
\end{equation}
where, for $s = r, r+l_i, \ldots , d$, $\omega_s$ is a quasi-homogeneous 1-form of degree $s$ (both $x$, $dx$ and
$y$, $dy$ have degree $l_j$ and $l_k$ respectively) 
and the contraction of $\omega_d$ with the weighted 
radial vector field $v=l_jx \frac{\partial}{\partial x}+l_ky\frac{\partial}{\partial y}$
is equal to zero. If $\omega_d$ = 0 then $i_v \omega_{d-l_i}\neq 0$ and the plane at 
infinity is invariant by the foliation defined by $\omega$. 
Note that the equation (\ref{E:decomposition}) is equivalent to $\omega$ is invariant by the action of $\mu_{l_i}$.  
When $\omega_r\neq 0$, we call $r$ the {\it algebraic
multiplicity of $\F$ at $p_i$}.

\section{Foliations with algebraic leaves on $\Proj$}\label{algebraicleaves}

\subsection{Logarithmic and rational foliations}
The following result has been proved, in a much more general context, 
by Bogomolov and McQuillan \cite{Bogo}. We will give another proof of the fact that  
the low degree foliations on $\Proj$ have infinitely many algebraic leaves. 
\begin{prop}\label{P:Rational} Let $\F$ be a saturated foliation of normal degree $d$ on $\Proj$
induced by $\omega$.  
If $d<|\w|$, then $\F$ is a rational fibration.

Furthermore, 
\begin{enumerate}
\item\label{E:Rat1} If $d=l_0+l_1$, then $\omega$ is conjugated to $l_1x_1dx_0-l_0x_0dx_1$. 
\item\label{E:Rat2} If $d=l_0+l_2$, then $\omega$ is conjugated to $l_2x_2dx_0-l_0x_0dx_2$.  
\item\label{E:Rat3} If $l_0+l_2< d$ and $d\neq l_1+l_2$, then $\omega$ is conjugated to
\[
(d-l_0)(x_2x_0^{i}+x_1^{j+1})dx_0-l_0x_0d(x_2x_0^{i}+x_1^{j+1}),
\] 
with the condition
$(i+1)l_0+l_2=(j+1)l_1+l_0=d$, $i,j\geq 1$. 
 \item\label{E:Rat4} If $d=l_1+l_2$, then $\omega$ is conjugated to 
\begin{enumerate} 
\item\label{E:first} $l_2x_2dx_1-l_1x_1dx_2$, when $1=l_0=l_1<l_2$  or $1<l_0<l_1<l_2$, 
\item\label{E:second} $l_2x_2dx_1-l_1x_1dx_2$ or 
$(l_2+l_1-1)(x_2x_0^{l_1-1}+x_1^{j+1})dx_0-x_0d(x_2x_0^{l_1-1}+x_1^{j+1})$, with the condition
$jl_1+1=l_2$, $j\geq 1$, 
when $1=l_0<l_1<l_2$. 
\end{enumerate}
\end{enumerate} 
\end{prop}

\begin{proof} 
First notice that $d\geq l_0+l_1$ and there are no saturated foliations on $\Proj$ 
if $l_0+l_1< d<l_0+l_2$. We now proceed to prove Items (\ref{E:Rat1}) - (\ref{E:Rat4}).

The assumption of Item (\ref{E:Rat1}) implies that $\omega=l_1x_1dx_0-l_0x_0dx_1$. 
Items (\ref{E:Rat1}) follows.
 
The assumption of Item (\ref{E:Rat2}) implies that $\omega=l_2x_2dx_0-l_0x_0dx_2$, 
if $1\leq l_0<l_1<l_2$, or 
$\omega=l_2x_2d(ax_0+bx_1)-(ax_0+bx_1)dx_2$, for some $a,b\in\C$, 
if $1=l_0=l_1<l_2$. Now Item (\ref{E:Rat2}) follows by performing a suitable 
change of coordinates. 

The assumption of Item (\ref{E:Rat3}) implies that 
$1\leq l_0<l_1<l_2$ and 
\begin{align*}
\omega=&x_0^{i}(l_2x_2dx_0- l_0x_0dx_2)+(x_0^i B(x_0,x_1)+bx_1^j)(l_1x_1dx_0-l_0x_0dx_1),\\
=&(l_2x_0^ix_2+l_1bx_1^{j+1}+l_1x_1x_0^{i}B)dx_0-l_0x_0(x_0^idx_2+bx_1^jdx_1+x_0^{i}Bdx_1),
\end{align*}
where $(i+1)l_0+l_2=(j+1)l_1+l_0=d$, $i,j\geq1$, $b\in\C^*$ and $B$ is of degree $d-(i+1)l_0-l_1$. 
To conjugate $\omega$ to
$(l_2x_0^ix_2+l_1bx_1^{j+1})dx_0-l_0x_0(x_0^idx_2+bx_1^jdx_1)$ we perform the change of coordinates
$\phi(x_0,x_1,x_2)=(x_0,x_1,x_2+H(x_0,x_1))$, where $H$ is chosen from the equation $H_{x_1}=B$.
Also this last 1-form can be written as 
\[
(d-l_0)(x_0^ix_2+\frac{b}{j+1}x_1^{j+1})dx_0-l_0x_0d(x_0^ix_2+\frac{b}{j+1}x_1^{j+1}).
\]
Up to automorphism of $\Proj$ we can suppose that $b=j+1$. Item (\ref{E:Rat3}) follows. 

The assumption of Item (\ref{E:Rat4}) together with Item (\ref{E:Rat2}) lets us assume that $1\leq l_0<l_1<l_2$. 
From Proposition \ref{P:Milnornumber} we see that $m(\F)=\frac{1}{l_0}$, so we must consider two cases: 
$l_0>1$ and $l_0=1$. In the first case the condition $l_0>1$ implies that $\sing(\F) = \{p_0\}$ and that 
\[
\omega=(l_2x_2dx_1-l_1x_1dx_2)+x_1C(x_0,x_1)(l_1x_1dx_0-l_0x_0dx_1),
\] 
where $C$ is of degree $l_2-(l_1+l_0)$. 
We now conjugate $\omega$ to
$l_2x_2dx_1-l_1x_1dx_2$ by making the change of coordinates 
$\phi(x_0,x_1,x_2)=(x_0,x_1,x_2+H(x_0,x_1))$, where $H$ is chosen from the equation $H_{x_0}=-x_1C$.
In the second case, $l_0=1$, up to an automorphism of $\Proj$ we can suppose that either 
$\sing(\F) = \{p_0\}$ or $\sing(\F) = \{p_2\}$. When $\sing(\F) = \{p_0\}$ the conclusion follows from the same argument we employed in the previous case $l_0> 1$. When $\sing(\F) = \{p_2\}$ we have
\begin{align*}
\omega=&x_0^{l_1-1}(l_2x_2dx_0-x_0dx_2)+(x_0^{l_1}A(x_0,x_1)+bx_1^{j})(l_1x_1dx_0-x_0dx_1),\\
=&(l_2x_0^{l_1-1}x_2+l_1bx_1^{j+1}+l_1x_1x_0^{l_1}A)dx_0-x_0(x_0^{l_1-1}dx_2+bx_1^jdx_1+x_0^{l_1}Adx_1),
\end{align*}
where $jl_1+1=l_2$, $j\geq 1$, $b\in\C^*$ and $A$ of degree $l_2-(l_1+1)$. 
Now we proceed as in Item (\ref{E:Rat3}) to obtain Item (\ref{E:second}). This completes the proof of Item (\ref{E:Rat4}), and therefore the proof of the proposition. 
\end{proof}

For the case of foliations of normal degree $|\w| $ we have the following result.

\begin{prop}\label{P:Logarithmic}   
 If $d=|\w|$, then the general element $\Fol_{d}$ is defined by a logarithmic $1$-form
with poles on three curves of degree $l_0$, $l_1$ and $l_2$. 

Furthermore, for $1\leq l_0<l_1<l_2$, a saturated foliation $\F\in \Fol_{|\w|}$ defined by $\omega$ 
is conjugated to one of the following 1-forms  
\begin{enumerate}
\item\label{E:Log1}
$x_0x_1x_2\left(a\frac{dx_0}{x_0}+b\frac{dx_1}{x_1}+c\frac{dx_2}{x_2}\right)$, 
where $a,b,c\in\C$ are such that $al_0+bl_1+cl_2=0$,  
\item \label{E:Log2}
$x_0^{i+1}x_1^{j+1}\left(l_1\frac{dx_0}{x_0}-l_0\frac{dx_1}{x_1}+d\left(\frac{x_2}{x_0^ix_1^j}\right)\right)$, 
where $i\geq 1$ and $j\geq 0$ are such that $il_0+jl_1=l_2$, 
\item \label{E:Log3}
$x_0^{l_1+1}x_2\left(l_2\frac{dx_0}{x_0}-\frac{dx_1}{x_1}+d\left(\frac{x_1}{x_0^{l_1}}\right)\right)$, 
when $1=l_0<l_1<l_2$.
\end{enumerate} 
\end{prop} 

\begin{proof} 
We will first prove that a general element $\F\in \Fol_{|\w|}$ is defined by 1-form conjugated to 
\begin{equation}\label{E:logarithmic}
x_0x_1x_2\left(a\frac{dx_0}{x_0}+b\frac{dx_1}{x_1}+c\frac{dx_2}{x_2}\right),
\end{equation}
with the condition $al_0+bl_1+cl_2=0$.   

In fact, by Proposition \ref{P:Milnornumber} we have $m(\F)=\frac{1}{l_0}+\frac{1}{l_1}+\frac{1}{l_2}$. 
Since $\F$ is a general element, up to an automorphism of $\Proj$ we can assume that 
$\sing(\F)=\{p_0,p_1,p_2\}$ and therefore $\F$ is induced by a 1-form $\eta$ equal to 
\begin{equation}\label{E:solloga}
-(ax_1x_2+l_1x_0x_1^2A(x_0,x_1))dx_0+
(bx_0x_2-l_0x_0^2x_1A(x_0,x_1))dx_1+cx_0x_1dx_2, 
\end{equation}
where $al_0+bl_1+cl_2=0$ and $A$ has degree $l_2-l_0-l_1$.  

Now we wish to conjugate $\eta$ to the 1-form (\ref{E:logarithmic}).  To do this we take a change 
of coordinates $\phi(x_0,x_1,x_2)=(x_0,x_1,x_2+F(x_0,x_1))$, where $F$ is chosen from the equation 
$aF+cx_0F_{x_0}=l_1x_0x_1A$. Since that $\F$ is general we can guarantee the existence of a solution 
$F(x_0,x_1)$, which is of degree $l_2$. 

Now let $\F\in \Fol_{|\w|}$  be a saturated foliation defined by $\omega$ and assume that $1\leq l_0<l_1<l_2$. 
Since $m (\F)=\frac{1}{l_0}+\frac{1}{l_1}+\frac{1}{l_2}$, two cases are to be considered: either $1<l_ 0$, in which  
case $\sing(\F)=\{p_0,p_1,p_2\}$, or if $l_ 0 = 1$, in which case, up to an automorphism of $\Proj$, one has  
$\sing(\F)=\{p_0,p_1,p_2\}$ or $\sing(\F)=\{p_1,p_2\}$. 

If $\sing(\F)=\{p_0,p_1,p_2\}$, then $\omega$ is as in (\ref{E:solloga}). 
Then we can apply the previous argument by noting that the existence of $F$ is equivalent to the fact that for each $i, j\geq 1$ with the condition $il_0 + jl_1 = l_2$ we have $a + ic\neq 0$. Item (\ref{E:Log1}) follows. In the opposite case there are $i, j\geq 1$ such that $il_0 + jl_1 = l_2$ and $a + ic = 0$, and therefore $b + cj = 0$. Then we can write $x_0x_1A$ in a unique way as $fx_0^ix_1^j+B$, where $f\in\C^*$ and we can apply the above argument to $B$ instead of $x_0x_1A$. Thus $\omega$ is conjugated to   
\[
-(-ix_1x_2+l_1fx_0^ix_1^{j+1})dx_0+
(-jx_0x_2-l_0fx_0^{i+1}x_1^{j})dx_1+x_0x_1dx_2. 
\]
Up to automorphism of $\Proj$ we can suppose that $f=1$, 
which corresponds to Item (\ref{E:Log2}).

When $l_0=1$ and $\sing(\F)=\{p_1,p_2\}$, up to an 
automorphism of $\Proj$, the 1-form $\omega$ is equal either to: 
\begin{equation}\label{E:firstpossi}
(-l_2x_1x_2+ux_0^{l_2}x_1+x_1A(x_0,x_1))dx_0+
(-ux_0^{l_2+1}-x_0x_1A(x_0,x_1))dx_1+x_0x_1dx_2, 
\end{equation}
 or to 
\begin{equation}\label{E:second possi}
(-l_1x_1x_2+ul_2x_0^{l_1}x_2+l_1x_1A(x_0,x_1))dx_0+
(x_0x_2-x_0A(x_0,x_1))dx_1-ux_0^{l_2+1}dx_2,   
\end{equation}  
where $u\in\C^*$ and $A$ is of degree $l_2$. 
If $\omega$ is equal to the 1-form in (\ref{E:firstpossi}) we apply the same previous argument of Item (\ref{E:Log2}), since $i = l_2$ and $j = 0$. If $\omega$ is equal to the 1-form in (\ref{E:second possi}), by means of a suitable change of coordinates we can suppose that $u=1$. Then we conjugate $\omega$ to the form 
$(-l_1x_1x_2+l_2x_0^{l_1}x_2)dx_0+x_0x_2dx_1-x_0^{l_2+1}dx_2$, 
which is in fact is equivalent to finding a change of coordinates $\phi(x_0,x_1,x_2)=(x_0,x_1,x_2+H(x_0,x_1))$, 
where $H$ is a solution of the equation $H-x_0^{l_1}H_{x_1} = -A$. This 1-form can be written as in Item (\ref{E:Log3}), thus completing the proof of the proposition.  
\end{proof}

The case $1=l_0=l_1<l_2$ of the proposition above will be dealt with in the next subsection, 
see Proposition \ref{P:Weighted special}.  

\subsection{Existence of algebraic leaves on $\Proj^2_k$}\label{S:Hirzebruch}
Recall that $\Proj^2_k=\Proj(1,1,k)$, $k\geq 2$, and 
its only singularity is $p_2=[0:0:1]\in U_2\subset \Proj^2_k$. The 
resolution of this singularity can be obtained from the map 
$\pi:\hat{U_2}\rightarrow U_2$, where 
$\hat{U_2}=\{((x,y)_{k},[s:t])\in U_2\times \Proj^1\,\,|\,xt=ys\}$ is a non-singular surface 
covered by two open sets $\hat{U_2}=V_0\cup V_1$ such that $V_i\simeq \C^2$ and 

$$\begin{array}{rclcrcl}
\pi_{|_{V_0}}:V_0&\rightarrow& U_2&\mbox{ and }&\pi_{|_{V_1}}:V_1&\rightarrow& U_2\\
(x,y)&\mapsto&(x^{1/k},x^{1/k}y)_k&\quad&
(x,y)&\mapsto&(xy^{1/k},y^{1/k})_k.
\end{array}$$
In the same way that $U_2$ is compactified as $\Proj^2_k$ we can compactify $\hat{U_2}$ 
as $\mb F _k$ and extend $\pi$ to $\pi:\mb F _k\rightarrow \Proj^2_k$ with the exceptional divisor 
$E=\pi^{-1}(\{p_2\})\simeq \Proj^1$ satisfying $E^2=-k$. See \cite[Section 2]{Artal} or 
\cite[Section 2.6]{Reid} for more details. 

\begin{prop}\label{P:Riccati} Let $\F$ be a saturated foliation of normal degree $d$ on
$\Proj^2_k$, $F$ be the strict transform of a line passing through $p_2$,  
$r$ be the algebraic multiplicity of $\F$ at $p_2$ and
$\G=\pi^*(\F)$ be the pull-back of the foliation $\F$. 
Then the following equality holds: 
\[
\NG=dF+\left(\frac{d-e}{k}\right)E,
\]
where $$e=\left\{\begin{array}{ll}
                 r-k&\quad \mbox{, if $E$ is $\G$-invariant,}\\
                 r&\quad \mbox{, if $E$ is not $\G$-invariant.}
                \end{array}\right.$$
Moreover $\G$ is a Riccati foliation with respect to the
natural rational fibration if and only if $e=d-2k$.
\end{prop}
\begin{proof} The first part follows from \cite[Proposition 7.3]{Artal}  and the computation of the total transform $\pi^*(\NF)$. 
The foliation $\G$ is transversal with respect
to the natural fibration if and only if $\NG\cdot F=2$, which is 
equivalent to $e=d-2k$.  
\end{proof}

\begin{cor}\label{C:Riccati} Under the same hypotheses of Proposition \ref{P:Riccati},   
$\G$ is a Riccati foliation with respect to the rational fibration and $E$ is $\G$-invariant  
 if and only if $r=d-k$.
 In particular, $\F$ has an invariant line.  
\end{cor}
\begin{proof}
The equivalence follows from Proposition \ref{P:Riccati}. We can also see that if $\G$ is a Riccati 
foliation and $E$ is $\G$-invariant, then the fiber through a singularity in $E$ is  
invariant by $\G$, thus $\F$ has an invariant line.
\end{proof}

In the notation of Proposition \ref{P:Riccati} we will say that $\F$ is a {\it Riccati foliation} if $\G$ is a 
Riccati foliation. 

The following proposition characterizes the low  degree foliations
on $\Proj^2_\p$ with some algebraic solution.

\begin{prop}\label{P:Weighted special} Any foliation of normal degree $d$ on $\Proj^2_k$, 
with $2\leq d\leq 2\p$ admits some invariant line. Furthermore
\begin{enumerate}
 \item \label{E:logRiccati} If $\p+1\leq d\leq 2\p$, then every saturated foliation in $\Fo$ is
 a Riccati foliation. 
 \item \label{E:logspecial} A saturated foliation $\F\in \Fol_{k+2}$ is defined by a 1-form conjugated to  
\begin{enumerate}
 \item \label{E:logspecial1} $(k+1)(x_0x_2+x_1^{k+1})dx_0-x_0d(x_0x_2+x_1^{k+1})$, 
when $\F$ has exactly one singularity and one invariant line. 
\item \label{E:logspecial2} $x_0^{2}x_2\left(k\frac{dx_0}{x_0}-\frac{dx_2}{x_2}+d\left(\frac{x_1}{x_0}\right)\right)$
 when $\F$ has exactly two singularities and one invariant line.
  \item \label{E:logspecial3} $x_0^{k+1}x_1\left(\frac{dx_0}{x_0}-\frac{dx_1}{x_1}+d\left(\frac{x_2}{x_0^{k}}\right)\right)$
 when $\F$ has exactly two singularities and two invariant lines.
 \item \label{E:logspecial4} $x_0^{i+1}x_1^{k-i+1}\left(\frac{dx_0}{x_0}-\frac{dx_1}{x_1}+d\left(\frac{x_2}{x_0^ix_1^{k-i}}\right)\right)$, 
 $0< i< k$, when $\F$ has exactly three singularities and two invariant lines.  
\item \label{E:logspecial5} $x_0x_1x_2\left(a\frac{dx_0}{x_0}+b\frac{dx_1}{x_1}+c\frac{dx_2}{x_2}\right)$, 
where $a+b+ck=0$, $a,b,c\in\C$, when $\F$ has exactly three singularities and three 
invariant curves (of degree $1$, $1$ and $k$). 
\end{enumerate} 

\end{enumerate}
\end{prop}
\begin{proof} First note that for $2<d\leq \p$, there are no saturated foliations. Now, 
we see that for $\p<d\leq 2\p$, a saturated foliation $\F\in \Fo$ 
admits an invariant line. In fact, $\F$ is given by a 1-form $\omega$ equal to 
\[
A(x_0,x_1)(\p x_2dx_0-x_0dx_2)+B(x_0,x_1)
(\p x_2dx_1-x_1dx_2)+C(x_0,x_1)(x_1dx_0-x_0dx_1),
\]
 where $A,B,C$ are quasi-homogeneous polynomials of
degree $d-\p-1$, $d-\p-1$ and $d-2$ respectively.
In the open set $U_2\simeq \C^2/\mu_{\p}$, we lift $\F|_{U_2}$ to
$\C^2$ and this lifting is given by
\[
\eta=\p A(x,y)dx+\p B(x,y)dy+C(x,y)(ydx-xdy).
\]
Since $\F$ is saturated we have $r=d-\p$, where $r$ is the algebraic multiplicity of $\F$ at $p_2$. 
Therefore by Corollary \ref{C:Riccati} $\F$ is a Riccati foliation and admits an invariant line. 
So Item (\ref{E:logRiccati}) follows. 

The assumption of Item (\ref{E:logspecial}) together with Proposition \ref{P:Milnornumber} imply that $m(\F)=2+\frac{1}{k}$. Thus up to an automorphism of $\Proj^2_k$ we can assume that $\sing(\F)=\{p_2\}$, or $\sing(\F)=\{p_1,p_2\}$ or 
$\sing(\F)=\{p_0,p_1,p_2\}$. Also, since the algebraic multiplicity of $\F$ at $p_2$ is $r=2$ the foliation 
$\F$ has at most two invariant lines. 

In the first case, $\sing(\F)=\{p_2\}$, our foliation $\F$ has just one invariant line. We can suppose 
this line is equal to $\{x_0=0\}$ and the proof follows by applying the same ideas we used in the proof of 
Proposition \ref{P:Rational}, Item (\ref{E:Rat3}). This proves Item (\ref{E:logspecial1}). 

In the second case, $\sing(\F)=\{p_1,p_2\}$, following the same ideas of Items (\ref{E:Log2}) and (\ref{E:Log3}) 
of Proposition \ref{P:Logarithmic} we get Item (\ref{E:logspecial2}) and (\ref{E:logspecial3}).

In the third case, $\sing(\F)=\{p_0,p_1,p_2\}$, using the same ideas from the proof of Items (\ref{E:Log1}) 
and (\ref{E:Log2}) of Proposition \ref{P:Logarithmic} we obtain Item (\ref{E:logspecial4}) and (\ref{E:logspecial5}).  
This completes the proof. 
\end{proof}

It is worth noting that in the general case $1\leq l_0 <l_1 <l_2$ using the same methods of this 
subsection we can show that a foliation of normal degree $d\leq 2l_2 + l_0-1$ on $\Proj$ 
admits an invariant curve.

\section{Foliations without algebraic leaves on $\Proj$}\label{S:wirthoutleaves}\label{Theorem}

\subsection{The universal singular set}
Set $X=\Proj\backslash \sing(\Proj)$. Recall that
$p_0=[1:0:0]$, $p_1=[0:1:0]$, $p_2=[0:0:1]$. We define the following
sets
\begin{align*}
   \Si(d):=&\{(x,\F)\in\,\Proj \times \Fo
   \,|\,x\in \sing(\F)\},\\
   \Si_X(d):=&\{(x,\F)\in\, X\times \Fo  |\,x\in \sing(\F)\},\\
   \Si_{p_i}(d):=&\{(p_i,\F)\in\,\Si(d)\},\,\,i=0,1,2.  
\end{align*}
We will see that the set $\Si_X(d)$ is irreducible, for this we need the following Lemma 
\cite[page 71]{Titu}, which we will use throughout this section. 
\begin{lemma}\label{L:sylvester}(Sylvester, 1894) For any positive integers $a$ and $b$ with 
$\gcd(a,b)=1$, define $g(a,b)$ to be the greatest positive integer $N$ for which the
equation
\begin{equation}\label{E:sylvester}
ax_1+bx_2=N,
\end{equation}
is not solvable in nonnegative integers. 
Then $g(a,b)=ab-a-b.$
\end{lemma}
\smallskip

\begin{prop}\label{P:irreducibleset} For all $d>l_1l_2$,
$\Si_X(d)$ is an irreducible subvariety and has codimension two in
$X\times\Fo$.
\end{prop}
\begin{proof} Consider the projection $\pi_1:\Si(d)\rightarrow
\Proj$. For every $x\in \Proj$, the fiber
$\pi_1^{-1}(x)$ is a subvariety of $\{x\}\times \Fo$,  
and it is isomorphic to a projective space, since if two
1-forms $\eta$ and $\eta'$ vanish at $x$, then the same is true for a
linear combination of them.

The action of the automorphism group of $\Proj$ on $X$ has four distinct orbits. 
In order to prove the proposition 
we will find 
for each of these orbits two $1$-forms which are
linearly independent at a point of any orbit. 

Since $d > l_1 l_2 \ge l_0 l_2$ 
we can apply Lemma \ref{L:sylvester} to obtain
positive integers $i_{1},j_{1}, i_{2}$, $j_{2}$ such that 
$i_{1} l_1 +j_{1} l_2 =i_{2} l_0 + j_{2}l_1=d$, and consequently 
$\alpha = x_1^{i_{1}-1} x_2^{j_{1}-1} ( l_2 x_2 dx_1 - l_1 x_1 dx_2)$ 
and $\beta = x_0^{i_{2}-1} x_1^{j_{2}-1} ( l_1 x_1 dx_0 - l_0 x_0 dx_1)$ 
belong $H^0(\mathbb P, \Omega^{[1]}_{\Proj}(d))$.

If $p=[1:1:1]$ then $V_p = < \alpha(p),\beta(p) > $, the
vector space generated by the evaluation of $\alpha$ and $\beta$ 
at $p$, has dimension $2$.

At the point $q_0=[0:1:1]$, the analogous vector space has dimension
two. But we can apply Lemma \ref{L:sylvester} to write $d -l_0 = i
l_1 + j l_2$ with $i$ and $j$ positive integers, and see that
$\delta_0 = x_1^{i -1} x_2^{j} ( l_0 x_0 dx_1 - l_1 x_1 dx_0 )$
belongs to $H^0(\mathbb P, \Omega^{[1]}_{\Proj}(d))$.
By considering the evaluation of $\delta_0$ and $\alpha$ at $q_0$, 
we see that they are $\mathbb C$-linearly independent. Applying the
same  argument to the others points $q_1= [1:0:1]$ and
$q_2=[1:1:0]$, we conclude that $\pi_1^{-1}(x) \subset X \times \Fo$ 
always have codimension two. It follows that $\Si_X(d)$ is
projective bundle over $X$, and therefore is smooth and irreducible.
\end{proof}

It is worth mentioning that in general the set $\Si(d)$ is not irreducible.

Recall that the sets $\Co_{n}(d)$, $\Dl_{n}(d)$ and $\Co^{p_i}_{n}(d)$, for each $i\in\{0,1,2\}$, 
are closed (see Lemma \ref{L:closedset}). So, we can state the following proposition. 

\begin{prop}\label{P:setcontradiction} Assume that  $d>l_1l_2$ and $\Co_{n}(d)=\Fo$ for some $n>0$.
\begin{enumerate}
 \item\label{P:condition1} If 
 $\Co^{p_i}_n(d)=\Co_n(d)$ for some $i$,
 then $\Si_{p_i}(d)\subset \Dl_n(d).$
 \item\label{P:condition2} If 
 $\Co^{p_i}_n(d)\neq \Co_n(d)$
 for every $i$,
 then $\Si_{X}(d)\subset \Dl_n(d).$
 \end{enumerate}
\end{prop}
\begin{proof} Item (\ref{P:condition1}) follows from definitions.  

We denote by $\pi$ the restriction of the natural projection of $\Si_X(d)$ to $\Fo$ and 
define the open set $U_n=\Fo\backslash\cup_{i=0}^2(\Co^{p_i}_n(d))$. The assumption of 
item (\ref{P:condition2}) implies that $U_n$ is a nonempty open set. 
Let $\F \in U_n$ be a saturated foliation. Then there exists a
curve $C$ of degree $n$ invariant by $\F$ that does not pass through $p_0$, $p_1$ and $p_2$.  
Since the map $\overline\varphi:\Proj^2\rightarrow \Proj$, given by $\overline\varphi([x_0:x_1:x_2])=[x_0^{l_0}:x_1^{l_1}:x_2^{l_2}]$, 
is generically finite and any curve on 
$\mathbb P^2$ has positive self-intersection we see that $C^2>0$. 
Using Camacho-Sad
Theorem \cite[page 5]{Brun2} $C^2=\sum_{p\in X \cap
C}CS(\F,C,p)$, it follows that there exists 
$p\in \sing(\F)\cap X$. Therefore
\[
\pi(\pi^{-1}(U_n)\cap\Dl_n(d))=U_n.
\]
By Proposition \ref{P:irreducibleset} we get $\pi^{-1}(U_n)$ is an
irreducible open set and $$\dim \pi^{-1}(U_n)=\dim \Fo.$$ 
Since $\dim\pi(\pi^{-1}(U_n)\cap\Dl_n(d))=\dim U_n=\dim \Fo$, we have  
\[
\pi^{-1}(U_n)\cap\Dl_n(d)=\pi^{-1}(U_n).
\]
Taking closure in $\Proj \times \Fo$, we get $\overline{\Si_X(d)}\cap\Dl_n(d)=\overline{\Si_X(d)}$.  
Items (\ref{P:condition2}) follows. 
\end{proof}

We see that when the conclusion of the above proposition is valid, we have that 
for each singularity of $\Proj$ that is also a singularity of each foliation $\F\in\Fo$ passes an invariant curve by $\F$ of degree $n$, or for each singularity in $X$ of each foliation $\F\in\Fo$ passes an invariant curve by $\F$ of degree $n$. 
In the next subsection in order to prove Theorem \ref{T:Weighted general} we will construct examples 
that contradict the conclusion of Items (\ref{P:condition1}) and (\ref{P:condition2}) of 
the above proposition for $d\gg 0$.

\subsection{Existence of singularities without algebraic separatrix}

First, we cons\-truct  a family of examples that contradict Item (\ref{P:condition2}) of
Proposition \ref{P:setcontradiction} for $d\gg 0$. The following example is
an adaptation of an example of J. V. Pereira,  see \cite[page
5]{Jorge}.
\bigskip

Let $\F_0$ be the foliation on $\Proj$ induced by the following 1-form
\[
\beta=x_0x_1x_2G\left(\lambda
l_1l_2\frac{dx_0}{x_0}+\mu l_0l_2\frac{dx_1}{x_1}+ +\gamma
l_0l_2\frac{dx_2}{x_2}-(\lambda+\mu+\gamma)\frac{dG} {G}\right),
\]
where $G(x_0,x_1,x_2)=x_0^{l_1l_2}+x_1^{l_0l_2}+x_2^{l_0l_1}$. Then
$\deg(\NFo)=l_0l_1l_2+l_0+l_1+l_2$. 

\begin{prop}\label{P:example1}
If $\lambda$, $\mu$ and $\gamma$  are $\Z$-linearly independent, then there are singularities of $\F_0$ 
in $\Proj \backslash \{x_0x_1x_2=0\}$ such that no invariant algebraic curves passes through them.
\end{prop}
\begin{proof} Observe that $\F_0=\phi^*\F$, where 
\begin{align*}
\phi: \Proj &\rightarrow \Proj^2 \\
\,[x_0:x_1:x_2] & \mapsto [x_0^{l_1l_2}:x_1^{l_0l_2}:x_2^{l_0l_1}], 
\end{align*}
and $\F$ is the foliation on $\Proj^2$ induced by 
$$\Omega=xyz(x+y+z)\left(\lambda
\frac{dx}{x}+\mu \frac{dy}{y}+ +\gamma
\frac{dz}{z}-(\lambda+\mu+\gamma)\frac{d(x+y+z)} {x+y+z}\right).$$
Then the proof follows from \cite[Lemma 2]{Jorge}.
\end{proof}
\begin{cor}\label{C:generalcase}
For all $d>l_0l_1l_2+l_0l_1+l_2$, there exists a foliation $\F$ on $\Proj$ with the following properties
\begin{itemize}
\item its normal $\Q$-bundle has degree $d$.
\item there are singularities of $\F$ in $\Proj \backslash \{x_0x_1x_2=0\}$ such that no $\F$-invariant algebraic curve passes through them. 
\end{itemize} 
\end{cor}
\begin{proof} Let $\F_0$ be the foliation on $\Proj$ of Proposition 
\ref{P:example1} induced by $\beta$. Note that, 
if $d>l_0l_1l_2+l_0l_1+l_2$, then $d-\deg(\NFo)>l_0l_1-l_0-l_1$. By 
Lemma \ref{L:sylvester} we can find $i,j>0$ such that  
$\omega=x_0^i x_1^j\beta$ belongs to 
$H^0(\Proj, \Omega^{[1]}_{\Proj}(d))$ 
and has the desired property. \end{proof}

\bigskip
Now, we are going to construct a family of examples that contradict
Item (\ref{P:condition1}) of Proposition \ref{P:setcontradiction} for $d\gg 0$.
\smallskip

For every $j_0=1,\ldots,l_2$,  let $j_1$ be the unique integer
satisfying $1\le j_1 \le l_2$ and $l_0 j_0 \equiv l_1 j_1 \pmod {l_2}$.
Consider the foliation $\F$ in $\C^2$ induced by the  1-form
\[
\eta=(x_1^{l_2}-1)x_0^{j_0-1}dx_0-a (x_0^{l_2}-1)x_1^{j_1-1}dx_1,
\]
in which $a\in \C\backslash \R$. 

\begin{lemma}\label{L:nondicriticalanddicrital}
 The foliation $\F$ does not have any $\F$-invariant algebraic curve
 passing through the point $(0,0)$.
\end{lemma}
\begin{proof} We have to consider two cases:

{\bf 1. First case}: $j_0=j_1=l_2$ (nondicrital case).
We extend $\F$ to a foliation  $\G$ on $\Proj^2$ which is induced by a 1-form 
$\omega$ equal to 
\[
(x_1^{l_2}-x_2^{l_2} ) x_0^{l_2-1} x_2 dx_0-a
(x_0^{l_2}-x_2^{l_2}) x_1^{l_2-1} x_2 dx_1 + (
a(x_0^{l_2}-x_2^{l_2}) x_1^{l_2} - (x_2^{l_2} - x_2^{l_2} )
x_0^{l_2} ) dx_2 \, .
\]
Thus
$\deg(\G)=2l_2-1$, 
and $\{x_2=0\}$, $\{x_1^{l_2}-x_2^{l_2}=0\}$,
$\{x_0^{l_2}-x_2^{l_2}=0\}$ are $\G$-invariants.
Notice that the singularities of $\G$ on
$\{x_1^{l_2}-x_2^{l_2}=0\}\cap\{x_0^{l_2}-x_2^{l_2}=0\}$ are
reduced.  Also, over each of these lines, $\G$ has only one extra
singularity corresponding to the intersection of the line with  $\{
x_2 = 0 \}$.

Suppose that there exists an algebraic curve $C$ invariant by $\G$
passing through $[0:0:1]$. B\'ezout's Theorem implies that $C$ must
intersect the line $\{x_1-x_2=0\}$. Since the singularities of $\G$
on this line outside of $\{x_2=0\}$ are all reduced and possess two
separatrices which do not pass through $[0:0:1]$, we conclude that $C$
intersects this line only at the point $[1:0:0]$. 
Let $\pi:M\rightarrow \Proj^2$ be the blow-up of $\mathbb P^2$ at the
point $[1:0:0]$, $E$ be the exceptional divisor, $\tilde{C}$ be the
strict transform of $C$, $L_2$ be the strict transform of
$\{x_2=0\}$, $F$ be the strict transform of
$\{x_1^{l_2}-x_2^{l_2}=0\}$, and $\tilde{\G}=\pi^*(\G)$ be the pullback
foliation on $M$. Then we see that the singularity at the point $[1:0:0]$ is
nondicritical, the singularities of $\tilde{\G}$ on $E$ are all
reduced and are contained in the intersection of $E$ with $L_2\cup F$. This 
claim contradicts the fact that $\tilde{C}\cap E$ is contained in $\sing(\tilde\G)$. 
This ends the first case. 

\smallskip
{\bf 2. Second case}: $l_2>j_0,j_1$ (dicritical case). 
 Again, we extend $\F$ to a foliation on $\Proj^2$ that
is denoted by $\G$. Assume, without loss of generality, that $j_0\ge j_1$. In 
this case the foliation $\G$ is induced by
\[
\omega =(x_1^{l_2}-x_2^{l_2} )x_0^{j_0-1} x_2 dx_0 -a
(x_0^{l_2}-x_2^{l_2} ) x_1^{j_1-1} x_2^{j_0-j_1+1} dx_1 + f dx_2 ,\,
\]
where $f(x_0,x_1,x_2)=a ( x_0^{l_2}-x_2^{l_2} ) x_1^{j_1}
x_2^{j_0-j_1} - ( x_1^{l_2} - x_2^{l_2} x_0^{j_0} )$. Thus
$\deg(\G)=l_2+j_0-1$, 
and $\{x_2=0\}$, $\{x_1^{l_2}-x_2^{l_2}=0\}$,
$\{x_0^{l_2}-x_2^{l_2}=0\}$ are $\G$-invariants.

Notice that the singularities of $\G$ on
$\{x_1^{l_2}-x_2^{l_2}=0\}\cap\{x_0^{l_2}-x_2^{l_2}=0\}$ are
reduced.  Also, over each of these lines, $\G$ has only one extra
singularity corresponding to the intersection of the line with  $\{
x_2 = 0 \}$.

Suppose that there exists an algebraic curve $C$ invariant by $\G$
passing through $[0:0:1]$. Let $\pi:M\rightarrow \Proj^2$ be the 
blow-up of $\mathbb P^2$ at
the points $[0:1:0]$ and $[1:0:0]$, $E=E_1\cup E_2$ be the
exceptional divisor, $\tilde{C}$ be the strict transform of $C$,
$L_2$ be the strict transform of $\{x_2=0\}$, $L_0$ be the strict
transform of $\{x_1-x_2=0\}$, $L_1$ be the strict transform of
$\{x_0-x_2=0\}$, and $\tilde{\G}=\pi^*(\G)$ be the pullback foliation on $M$.
Then we have the singularities at the points $[1:0:0]$ and
$[0:1:0]$ are dicritical and the singularities of $\tilde{\G}$ on $E$
do not belong to the lines $L_1$ and $L_2$.

Now we define the following map
\begin{align*}
\phi:Pic(M)&\rightarrow \Z^2\\
D\quad&\mapsto(D.L_0,D.L_1).
\end{align*}
Observe that $\phi$ is a surjective map and $\ker(\phi)=\Z L_2.$
By the previous discussion we see that $\tilde{C}\in \ker(\phi)$, hence we  
can write $\tilde{C}=bL_2$ in $Pic(M)$ for some $b\in\Z$. This is a
contradiction, since $\tilde{C}$ has positive self-intersection.  \end{proof}

\begin{cor}\label{C:particularcase} For all $d\geq l_2l_1+l_2l_0+l_2$, all $n\in\N$ and
 every $i=0,1,2$, we have
$$\Co^{p_i}_n(d)\neq \Fo.$$
Furthermore, if $l_0=l_1=1$ and $l_2\geq 2$, then $\Co^{p_2}_n(d)\neq \Fo$,
for all $d\geq 2l_2+1$ and all $n\in\N $.
\end{cor}
\begin{proof} We just show that $\Co^{p_2}_n(d)\neq \Fo$, and similar
arguments can be used for the other cases. 
Take $d\geq l_2l_1+l_2l_0+l_2$, let $j_0$, $j_1$ be the unique
integers satisfying $1\leq j_0,j_1\leq l_2$ and $d\equiv
l_0j_0\equiv l_1j_1 \pmod {l_2}$. Then Lemma \ref{L:nondicriticalanddicrital} and 
implies that the the foliation $\F$ on $\C^2$ given by
$$\eta=(x_1^{l_2}-1)x_0^{j_0-1}dx_0-a (x_0^{l_2}-1)x_1^{j_1-1}dx_1,$$
does not have any $\F$-invariant algebraic curve passing through the
point $(0,0)$.

We can extend $\F$ to a foliation
$\hat{\F}$ on $\Proj$, which is induced by $\theta$ and
\[
\deg(N\hat{\F})=\left\{
\begin{array}{cc}
l_0j_0+l_1l_2+l_2&\mbox{, if $l_0j_0\geq l_1j_1$,} \\
l_1j_1+l_0l_2+l_2&\mbox{, if $l_1j_1>l_0j_0$.}
\end{array}
\right.
\]
Since $d\equiv \deg(N \hat{\F}) \mod l_2$ and $d\geq
l_2l_1+l_2l_0+l_2$, we can multiply the 1-form $\theta$ by an
adequate power of $x_2$ and construct a foliation $\Ho$ on
$\Proj$ with normal $\Q$-bundle of degree $d$. The foliation
$\Ho$ does not have any $\Ho$-invariant algebraic curve passing
through the point $[0:0:1]$. Hence $\Ho\notin \Co^{p_2}_n(d)$.

If $l_0=l_1=1$ and $d\geq 2l_2+1$, then $d\equiv j \pmod {l_2}$, for
a unique integer  $1\leq j \leq l_2$. Taking $j=j_0=j_1$, the
foliation $\hat{\F}$ constructed above satisfies
$\deg(N\hat{\F})=2l_2+j$. By a same procedure we can construct a
foliation $\Ho\notin \Co^{p_2}_n(d)$ for all $n\in \N$.
\end{proof}

\subsection{Proof of Theorem \ref{T:Weighted general}}
By Lemma \ref{L:closedset}, 
$\Co_n(d)$ is an algebraic closed subset of $\Fo$ for all
$n\in\, \N$. We claim that if $d\geq l_0l_1l_2+l_0l_1+2l_2$, then
$\Co_n(d)\neq \Fo$, for all $n\in\,\N$. In fact, if we had
$\Co_n(d)=\Fo$ for some $n\in\,\N$, then we could apply
Proposition \ref{P:setcontradiction} and this would contradict Corollary
\ref{C:generalcase} and Corollary \ref{C:particularcase}. 
By applying Baire's Theorem we finish the proof. \qed

\begin{cor} Under the assumptions of Theorem \ref{T:Weighted general}, 
a generic $G$-invariant foliation of degree $d$ on $\Proj^2$ does not admit any invariant algebraic curve if
$d\geq l_0l_1l_2+l_0 l_1+ 2l_2-2$.

\end{cor}

\begin{proof} It follows from Theorem \ref{T:Weighted general} and Lemma \ref{L:isomorphism}.
\end{proof}

In general, for $l_0, l_1, l_2$ positive integers not necessarily pairwise coprimes, a reduction method 
is known to obtain pairwise coprimes as follows:

Let $d_0=\mdc(l_1,l_2)$, $d_1=\mdc(l_0,l_2)$, $d_2=\mdc(l_0,l_1)$, 
$a_0=\lcm(d_1,d_2)$, $a_1=\lcm(d_0,d_2)$, $a_2=\lcm(d_0,d_1)$ and $a=\lcm(a_0,a_1,a_2)$. 
Then $l'_i=l_i/a_i$, $i=0,1,2$ are pairwise coprimes. 
Set $\Proj:=\Proj(l_0,l_1,l_2)$ and $\Proj'=\Proj(l'_0,l'_1,l'_2)$. 
The map 
\begin{align*}
\Psi: \Proj&\rightarrow \Proj'\\
\,[x_0:x_1:x_2] \quad &\mapsto [x_0^{d_0}:x_1^{d_1}:x_2^{d_2}],
\end{align*}
is an isomorphism, see \cite[Subsection 1.3]{Dolgachev}. This map induces an isomorphism 
\[
\Psi^*:
\Proj H^0(\Proj',\Omega^{[1]}_{\Proj'}(d))\to
\Proj H^0(\Proj,\Omega^{[1]}_{\Proj}(ad)).
\]

We note that in the case $a>1$, that is, $l_0, l_1, l_2$ are not pairwise coprimes,   
if the normal degree of a saturated foliation 
$\F$ on $\Proj(l_0,l_1,l_2)$ is not a multiple of $a$, then $\F$ has an invariant curve, see 
\cite[Proposition 2.3.11]{Lizarbe14} for more details. From these observations and using Theorem 
\ref{T:Weighted general} we can obtain the following.  
\begin{cor} Under the conditions stated above,  
a generic invariant foliation with normal $\Q$-bundle of degree $ad$ on $\Proj(l_0,l_1,l_2)$ does not 
admit any invariant algebraic curve if 
\[
d\geq \frac{l_0l_1l_2}{a_0a_1a_2}+\frac{l_0 l_1}{a_0a_1}+ \frac{l_0l_2}{a_0a_2}+\frac{l_1l_2}{a_1a_2}.
\]
\end{cor}


In the next subsection, we will construct counterexamples to prove Theorem
\ref{T:Weighted special}.

\subsection{Existence of singularities without algebraic separatrix on $\Proj^2_\p$}
The following family of the examples allows us to obtain the bound of
Theorem \ref{T:Weighted special}.

\bigskip

Let $\F_1$ be the foliation on $\Proj^2_\p$, $\p>1$, induced
by the following 1-form
\[
\delta=-\p x_2(x_2-x_0x_1^{\p-1})dx_0+\p x_0x_2(x_1^{\p-1}-
x_0x_1^{\p-2})dx_1+x_0(x_2-x_1^{\p})dx_2.
\] 
Notice that
$\deg(\NFone)=2\p+1$, $\sing(\F_1)=\{[0:1:0],\,[1:0:0],\,[1:1:1]\}$
and $\{x_0=0\}\cap \sing(\F_1)=\{[0:1:0]\}$.
\begin{lemma}\label{L:excepcionalcase}
 The foliation $\F_1$ does not have any invariant algebraic curve
 passing through the point $[1:1:1]$ .
\end{lemma}
\begin{proof} Observe that the lines $\{x_0=0\}$ and $\{x_2=0\}$ are
$\F_1$-invariant. Suppose that there exists an algebraic curve $C$
invariant by $\F_1$ passing through $[1:1:1]$. Since $\{x_0=0\}\cap
\sing(\F_1)=\{[0:1:0]\}$, using B\'ezout's Theorem for weighted
projective planes \cite[Proposition 8.2]{Artal} 
we conclude that $\{x_0=0\}$ only intersects $C$
at the point $[0:1:0]$. 
Note that $[0:1:0]$ is a saddle-node singularity with only two
separatrices $\{x_0=0\}$ and $\{x_2=0\}$. In particular, $C$ must be one of these lines, 
contradicting $[1:1:1] \in C$. 
\end{proof}


\subsection{Proof of Theorem \ref{T:Weighted special}}

It is enough to show that $\Co_n(d)\neq \Fo$, for all $n\in\,\N$ and all
$d\geq 2\p+1$. 
Reasoning by contradiction,
suppose that this does not hold. Then Proposition \ref{P:setcontradiction} and
Corollary \ref{C:particularcase} imply $\Si_{X}(d)\subset \Dl_n(d)$, 
that contradicts Corollary \ref{C:generalcase}, for $d>2\p+1$, and Lemma
\ref{L:excepcionalcase}, for $d=2\p+1$. 
\qed

\begin{remark} In contrast with Theorem \ref{T:Weighted general}, in Theorem \ref{T:Weighted special} 
we obtain the best possible bound $d\geq 2\p+1$. 
\end{remark}

\section{Holomorphic foliations on Hirzebruch surfaces}\label{Hirzebruch}

Let $\mb F_{\p}=\Proj(\Ol_{\Proj^1}\oplus\Ol_{\Proj^1}(\p))$ be the 
Hirzebruch surface, $\pi:\mb F_\p \rightarrow \Proj^2_\p$ be
the minimal resolution of $\Proj^2_k$, $E$ be
the exceptional divisor and $F$ be the strict transform of a line passing through $p_2$, 
see Section \ref{S:Hirzebruch}. Observe that $\Pic(\mb F_\p)=\Z F\oplus \Z
E,$ in which $E\cdot E=-\p$, $F\cdot F=0$ and $F\cdot E=1$.

We denote by $\Ro(a,b):=\Proj
H^0(\mb F_\p,\Omega^1_{\mb F_\p}\otimes\Ol_{\mb F_\p}(aF+bE))$ the
space of holomorphic foliations with normal bundle of bidegree
$(a,b)$ on $\mb F_\p$.
\begin{remark} 
Let $\F$ be a foliation with normal $\Q$-bundle of degree $d$ on
$\Proj^2_\p$, $r$ is the algebraic multiplicity of $\F$ at $p_2$ and $\G=\pi^*\F$ 
be the foliation on $\mb F_{\p}$. Recall that $\NG=\Ol_{\mb F_{\p}}\left(dF+\left(\frac{d-e}{\p}\right)E\right),$ 
where
\[
e=\left\{\begin{array}{cl}
     r-\p& \mbox{, if $E$ is $\G$-invariant,}\\
     r&\mbox{, if $E$ is not $\G$-invariant.}
    \end{array}\right.
\]    
\end{remark}

We shall now prove that under suitable conditions on $(a,b)$ 
each foliation with bidegree $(a,b)$ on $\mb F_\p$ has algebraic solutions.

\begin{prop}\label{P:Hircurve} If $a<b\p+2$ or $b<3$, then any foliation $\G\in \Ro(a,b)$ admits
some invariant algebraic curve.
\end{prop}
\begin{proof} 
We first prove that $E$ is $\G$-invariant when $a<b\p+2$. Suppose it is not. 
Then, by the tangency formula,
$0\leq Tang(\G,E)=-b\p+a-2$. So, $a\geq b\p+2$, contradicting the hypothesis. 
We can thus assume that $a-b\p\geq 2$. If $b=0$ then the foliation $\G$ is the 
rational fibration, so we may also assume   
$0<b\leq 2$. Baum-Bott formula implies that $\sum
BB(\NG,p)=\NG^2=(aF+bE)^2=b(a+a-b\p)>0$, and therefore there exists a
point $p\in \sing(\G)$.  We claim that the fiber 
$F$ 
passing through the point $p$ 
is $\G$-invariant. If this is not the case, 
then 
the tangency formula implies 
\[
0<Tang(\G,F)=\NG.F-\chi(F)=b-2\leq 0,
\]
which is again a contradiction. This completes the proof of the proposition.
\end{proof}

Let $\chi\in \Q[t]$. Define two subsets of $\mb F_\p\times \Ro(a,b)$
by
\[
\Si(a,b)=\{(x,\G)\in \mb F_\p\times \Ro(a,b)|\,x \in \sing(\G)\},
\]
and
\begin{align*} 
\Dl_{\chi}(a,b)=\{
(x,\G)\in \mb F_\p \times \Ro(a,b)|
\, \mbox{ $x$ is in some subscheme,} & \mbox{ invariant by $\G$},
\\
&\mbox{of Hilbert
polynomial $\chi$}
\}.
\end{align*}
\begin{prop}\label{P:Hirreducibleset} The following statements hold.
\begin{enumerate}
 \item If $b\geq 2$ and $a \geq b\p+2$, then $\Si(a,b)$ is a closed irreducible variety
  of $\mb F_\p \times \Ro(a,b)$ and $\dim_{\C}\Si(a,b)=\dim_{\C}\Ro(a,b).$
 \item $\Dl_{\chi}(a,b)$ is a closed subset of $\mb F_\p\times \Ro(a,b)$.
\end{enumerate}
 \end{prop}
\begin{proof} 1. Let $\Sigma(a,b)$ denote the line bundle
$\Ol_{\mb F_\p}(aF+bE)$ on $\mb F_\p$. Consider the following exact
sequence of sheaves
\[
\xymatrix{0\ar[r]&\ker \psi \ar[r]^-{i}&
H^0(\mb F_\p,\Omega^1_{\mb F_\p}\otimes \Sigma(a,b)) \otimes
\Ol_{\mb F_\p}\ar[r]^-{\psi}& \Omega^1_{\mb F_\p}\otimes
\Sigma(a,b)},
\]
where $\psi(x,\eta)=\eta (x)$. Notice that 
if $\ker \psi$ is a vector bundle then $\Proj(\ker \psi)=\Si(a,b)$ is 
an irreducible variety of codimension two on $\mb F_\p\times \Ro(a,b)$. 
Thus, it suffices to prove $\ker
\psi$ is a vector bundle, which we do now. Since
\[
\pi_{|_{\mb F_\p\backslash E}}:\mb F_\p\backslash E\rightarrow\Proj^2_\p\backslash \{[0:0:1]\},
\] is an isomorphism and
the automorphism group of $\Proj^2_\p$ acts transitively on
$\Proj^2_\p\backslash \{[0:0:1]\}$, then the automorphism
group of $\Proj^2_\p$ acts transitively on
$\mb F_\p\backslash E$. Therefore $\ker\psi_x$ has dimension
$\dim_{\C}\Ro(a,b)-2$ for all $x\in \mb F_\p\backslash E$. We claim
that $\ker\psi_x$ also has dimension
$\dim_{\C}\Ro(a,b)-2$ for all $x\in E$. Indeed, we take two foliations $\F_1$,
$\F_2$ on $\Proj^2_\p$ with normal $\Q$-bundle of degree
$a$ induced by the 1-forms $\omega=x_2^b
A(x_0,x_1)(x_1dx_0-x_0dx_1)$ and
$\eta=x_2^{b-2}B(x_0,x_1)(\p x_2dx_0-x_0dx_2)$ respectively, where
$A$ and $B$ are homogeneous polynomials of degree $a-b\p-2$ and
$a-(b-1)\p-1$ that can be chosen such that the forms $\tilde{\omega}$ 
and $\tilde{\eta}$ defining  
$\pi^*(\F_1), \pi^*(\F_2)\in\Ro(a,b)$ are linearly independent at $x\in E$. 
Then any 1-form $\alpha\in \Ro(a,b)$ with  
$\alpha(x)\neq0$ can be written as a $\C$-linear combination of
$\tilde{\omega}$ and $\tilde{\eta}$.
Thus $\dim_{\C} \ker \psi_x $ is constant as a function
of $x\in X$, hence $\ker \psi$ is a vector bundle. 

2. Follows directly from \cite[Lemma 5.1, page 9]{Cou}.
\end{proof}

\begin{lemma}\label{L:setsingular_nonempty} Let $\G\in \Ro(a,b)$ be a foliation on $\mb F_\p$ and $C$
be an algebraic curve invariant by $\G$. If $b\geq3$ and $a\geq b\p+2$, then 
\[
C\cap \sing(\G)\neq\emptyset.
\]
\end{lemma}
\begin{proof} We consider two cases:

1. If $C=E$ in $\Pic(\mb F_\p)$, then the Camacho-Sad formula gives $C^2=-\p$, which implies
 that $C\cap \sing(\G)\neq\emptyset$.

2. Suppose $C=mF+nE$ in $\Pic(\mb F_\p)$, with $m>0$,
$n\geq0$, and assume for contradiction that 
$C\cap \sing(\G)=\emptyset$. Then 
 $C^2=0$ by the Camacho-Sad formula, and from the vanishing formula
we obtain 
\[
\NG\cdot C=C^2+Z(\G,C)=0.
\]
But $\NG\cdot C=n(a-b\p)+bm>0$, and this is a contradiction.
\end{proof}

\subsection{Proof of Theorem \ref{T:Hirzebruch}}
Consider the second projection $\pi_2:\Si(a,b)\rightarrow \Ro(a,b)$
and fix a polynomial $\chi\in \Q[t]$ of degree one. Suppose that
$\pi_2(\Dl_{\chi}(a,b))=\Ro(a,b)$; that is, every foliation of
bidegree $(a,b)$ has an algebraic invariant curve with Hilbert
polynomial $\chi$. By Lemma \ref{L:setsingular_nonempty}, we have 
\[
\pi_2(\Dl_{\chi}(a,b)\cap \Si(a,b))=\pi_2(\Si(a,b)).
\]
Since $\Si(a,b)$ is an irreducible variety and
$\dim_{\C}\Si(a,b)=\dim_{\C} \Ro(a,b)$, by Proposition \ref{P:Hirreducibleset}
we get $\Si(a,b)\cap \Dl_{\chi}(a,b)=\Si(a,b)$.

To finish the proof we take the foliation $\F_1$ on
$\Proj^2_\p$ of Lemma \ref{L:excepcionalcase} 
which is induced by $\delta$ and has degree $2\p+1$. Let $\F$ be the foliation on
$\Proj^2_\p$ with normal $\Q$-bunlde of degree $a$ induced
by $\omega=x_0^{a-b\p+\p-1}x_2^{b-3}\delta$.
Note that the foliation $\G=\pi^*(\F)$ on $\mb F_\p$ has
bidegree $(a,b)$, and there is no invariant curve passing through the point 
$\pi^{-1}([1:1:1])$, which is
a singularity of $\G$. 
This is a contradiction. Because there are only countably many Hilbert
polynomials, this completes the proof the theorem.
\qed

\section*{Acknowledgements}
The author wishes to express his deepest gratitude to Jorge Vit\'orio Pereira for lots 
of fruitful discussions about the content of this paper.
 The author also wishes to thank Maycol Falla Luza and Raphael Constant da Costa 
 for the suggestions and comments on the ma\-nus\-cript. 
 This work was developed at IMPA (Rio de Janeiro, Brazil) and partially supported by CNPq and CAPES.

\bibliography{references}{}
\bibliographystyle{plain}
\end{document}